\documentclass{article}
\usepackage{amsfonts}

\usepackage{amsmath}
\usepackage{afterpage}
\usepackage[T2A]{fontenc}
\usepackage{indentfirst}
\usepackage{graphicx}
\usepackage{xcolor}

\def\qed{\hbox to 0 pt{}\hfill$\rlap{$\sqcap$}\sqcup$}
\newtheorem{theorem}{Theorem}

\newtheorem{corollary}[theorem]{Corollary}

\newtheorem{definition}[theorem]{Definition}
\newtheorem{example}[theorem]{Example}

\newtheorem{lemma}[theorem]{Lemma}

\newtheorem{remark}[theorem]{Remark}

\newenvironment{proof}[1][Proof]{\textbf{#1.} }{\ \rule{0.5 em}{0.5 em}}

\begin{document}

\title{\textbf{{\Large {Green's Functions and Existence of Solutions of Nonlinear Fractional Implicit Difference Equations with Dirichlet Boundary Conditions}}}}
\author{{Alberto Cabada$^{a,b}$, Nikolay D. Dimitrov$^c$ and Jagan Mohan Jonnalagadda$^d$%
} \\
$^a$CITMAga, 15782 Santiago de Compostela, Galicia, Spain.\\ 
$^b$ Departamento de Estat\'{\i}stica, An\'alise Matem\'atica e Optimizaci\'on, \\
	Facultade de Matem\'aticas,
	Universidade de Santiago de Compostela, \\Galicia, Spain, \, alberto.cabada@usc.es \\
$^c$Department of Mathematics, University of Ruse, 7017 Ruse,\\
Bulgaria; \, ndimitrov@uni-ruse.bg\\
$^d$ Department of Mathematics, Birla Institute of Technology\\
and Science Pilani, Hyderabad - 500078, Telangana, India;\\
\, {j.jaganmohan@hotmail.com}}
\maketitle


\begin{abstract}
This article is devoted to deduce the expression  of
the Green's function related to a general constant coefficients fractional
difference equation coupled to Dirichlet conditions. In this case, due to the points where some of the fractional operators is applied, we are in presence of a implicit fractional difference equation. Such property makes it more complicated to calculate and manage the expression of the Green's function. Such expression, on the contrary to the explicit case where it follows from finite sums, is deduced from series of infinity terms. Such expression will be deduced from the Laplace transform on the time scales of the integers. Finally, we prove two existence results for nonlinear problems, via suitable fixed point theorems.
\end{abstract}


\vskip 0.25 in

\vskip 0.25 in


\vskip 0.25 in


\vskip 0.25 in

\section{Introduction and Preliminaries}

Considering integrals and derivatives of arbitrary orders allows modeling
many real phenomena in which the value that the solution takes at a given
instant depends on the value of the solution in all the previous moments of
the process. Thus, fractional calculus became very useful in several fields
as viscoelasticity, neurology and control theory \cite{heymans, kilbas,
magin, podlubny, samko, west}. During the last decades, a lot of authors
studied fractional difference equations and there has been a progress made
in developing the basic theory in this field. We refer to the reader the
monographs \cite{goodrich2,miller} for more details. We use the standard
notation $\mathbb{N}_{a}=\left\{ a,a+1,a+2,\dotsc \right\} $ for $a\in 
\mathbb{R}$, and $[c,c+n_0]_{\mathbb{N}_{c}}=[c ,c+n_0] \cap \mathbb{N}_{c}$%
, for $c\in \mathbb{R}$ and $n_0 \in {\mathbb{N}_{1}}$.

In \cite{atici3} Atici and Eloe proved that for all $(t,s) \in
[\upsilon-2,\upsilon+b+1]_{\mathbb{N}_{\upsilon-2}} \times [0,b+1]_{\mathbb{N%
}_{0}}$, the following function 
\begin{equation*}
G_{0}(t,s)=\frac{1}{\Gamma \left( \upsilon \right) }\left\{ 
\begin{array}{ll}
\frac{t^{\left( \upsilon -1\right) }\left( \upsilon +b-s\right) ^{\left(
\upsilon -1\right) }}{\left( \upsilon +b+1\right) ^{\left( \upsilon
-1\right) }}-\left( t-s-1\right) ^{\left( \upsilon -1\right) }, & s<
t-\upsilon +1, \\ 
\frac{t^{\left( \upsilon -1\right) }\left( \upsilon +b-s\right) ^{\left(
\upsilon -1\right) }}{\left( \upsilon +b+1\right) ^{\left( \upsilon
-1\right) }}, & t-\upsilon +1\leq s,
\end{array}
\right.
\end{equation*}
is the related Green function to the Dirichlet problem \vskip -10pt 
\begin{equation*}
\begin{split}
-\Delta ^{\upsilon }y\left( t\right) & =h( t+\upsilon -1) , \; t \in
[0,b+1]_{\mathbb{N}_{0}}, \\
y\left( \upsilon -2\right) & =y\left( \upsilon +b+1\right) =0,
\end{split}
\end{equation*}
with $\upsilon \in \mathbb{R}$, $1<\upsilon <2$ and $b\in \mathbb{N}$.
Moreover, they proved that $G_0(t,s) >0$ for all $(t,s) \in
[\upsilon-1,\upsilon+b]_{\mathbb{N}_{\upsilon-1}} \times [0,b+1]_{\mathbb{N}%
_{0}}$.

Using previous expression and constructing Green's function as a series of
functions, in \cite{CD}, by using the spectral theory is ensured, for a
suitable range of values of the nonconstant function $a(t)$, the
positiveness of the Green's function related to the following Dirichlet
problem 
\begin{equation*}
\begin{split}
-\Delta ^{\upsilon }y\left( t\right) +a\left( t+\upsilon-1\right) y\left(
t+\upsilon-1\right) & =h\left( t+\upsilon -1\right), \\
y\left( \upsilon -2\right) & = y\left( \upsilon +b+1\right) =0,
\end{split}
\label{mainprob}
\end{equation*}
\vskip -3pt \noindent for $t\in [0,b+1]_{\mathbb{N}_{0}}$, where $\upsilon
\in \mathbb{R}$ with $1<\upsilon <2$ and $b\in \mathbb{N}$, $b \ge 5$.

A similar approach has been done in \cite{BCW} for the following problem with mixed conditions:
\begin{equation*}
	\begin{split}
		-\Delta ^{\upsilon }y\left( t\right) +a\left( t+\upsilon-1\right) y\left(
		t+\upsilon-1\right) & =h\left( t+\upsilon -1\right), \\
		y\left( \upsilon -2\right) & =\Delta^\beta y\left( \upsilon +b+1-\beta\right) =0,
	\end{split}
\end{equation*}
with $1<\upsilon \le 2$ and $0\le \beta \le 1$.

Using another approach in \cite{atici, atici2, AE} the general expression of
several linear $n$-th order initial value problems is obtained. They use $%
R_{0}(f(t))(s)$ the Laplace transform on the time scale of integers \cite
{bohner, donahue}, which is defined by the following expression: 
\begin{equation*}
R_{t_{0}}(f(t))(s)=\sum_{t=t_{0}}^{\infty}\left(\frac{1}{s+1}\right)^{t+1}
f(t) .
\end{equation*}

Recently, in \cite{CDJ} the authors considered the problem 
\begin{eqnarray}
- \Delta^{\upsilon}y(t) + \alpha \Delta^{\mu}y(t + \upsilon - \mu - 1) &=&
h(t + \upsilon - 1),  \label{LDE} \\
y\left( \upsilon -2\right) & =& y\left( \upsilon +b+1\right) =0,
\label{e-Dirichlet}
\end{eqnarray}
for $t \in I \equiv [0,b+1]_{\mathbb{N}_{0}}$, where $\mu$, $\upsilon \in 
\mathbb{R}$ such that $0 < \mu < 1$ and $1 < \upsilon < 2$; $%
\Delta^{\upsilon}$ and $\Delta^{\mu}$ are the standard $\upsilon$-th and $\mu
$-th order Riemann--Liouville fractional difference operators, respectively; 
$\alpha $ is a real constant and $h : I \rightarrow \mathbb{R}$.

By using the Laplace transform $R_{0}(f(t))(s)$ they obtained the general
expression of equation \eqref{LDE} and deduced the explicit expression of the
Green's function related to problem \eqref{LDE}--\eqref{e-Dirichlet}. It was
proven that such Green's function has some symmetric properties and is
positive on $[\upsilon-1,\upsilon+b]_{\mathbb{N}_{\upsilon-1}} \times
[0,b+1]_{\mathbb{N}_{0}}$ for all $\alpha \ge 0$ and $\upsilon - \mu - 1 > 0$%
, which improved the results given in \cite{atici3}. Moreover, the authors
deduced some strong positiveness conditions on the Green's function that
allow them to construct suitable cones where to deduce the existence of
solutions of related nonlinear problems. We point out that, in such case, the fact that the fractional operator $\Delta^\mu$  is defined on the points $t+\upsilon-\mu-1$ gives us a explicit equation. Such property gives us the expression of the Green's function as a combination of finite sums.

The aim of this paper is to continue our work in this direction as we
consider the following equation 
\begin{equation}  \label{LDE 1}
- \Delta^{\upsilon}y(t) + \alpha \Delta^{\mu}y(t + \upsilon - \mu) = h(t +
\upsilon - 1),\quad t \in I \equiv \left\{0, 1, \dotsc, b + 1\right\},
\end{equation}
coupled to the boundary conditions \eqref{e-Dirichlet}. 

Here $\mu$, $%
\upsilon \in \mathbb{R}$ such that $0 < \mu < 1$ and $1 < \upsilon < 2$; $%
\Delta^{\upsilon}$ and $\Delta^{\mu}$ are the standard $\upsilon$-th and $\mu
$-th order Riemann--Liouville fractional difference operators, respectively; 
$\alpha $ is a constant and $h : I \rightarrow \mathbb{R}$. We point out that even if we use Laplace transform $R_{0}(f(t))(s)$
to equation \eqref{LDE 1}, we deduce that the sums are not finite as ones
given in \cite{CDJ}. As a result, we study the convergence of the series and
we apply some fixed point to deduce some existence results for a related non
linear problem. We remark that problem \eqref{LDE 1} coupled to the Dirichlet conditions \eqref{e-Dirichlet} has been studied in \cite{henderson} for the particular case of $\alpha=0$.

The paper is organized as follows: After an introduction where we compile the main concepts and properties the we will use along the paper, we obtain, in Section 2, the expression of the Green's function related to problem \eqref{LDE}--\eqref{e-Dirichlet} for $\left|\alpha\right|<1$ (which is the condition that characterizes the convergence of the used series). In Section 3 we deduce two existence results for nonlinear problems. Such existence results follow from the expression of the Green's function by constructing an operator whose fixed points coincides with the solutions of the problems that we are looking for. We finalize the paper with two examples that point out the applicability of the obtained existence results. 

First we recall some basic definitions and lemmas, which will be used till
the end of this work.

\begin{definition}
We define $t^{\left( \upsilon \right) }=\frac{\Gamma \left( t+1\right) }{
\Gamma \left( t+1-\upsilon \right) }$, for any $t$ and $\upsilon $ for which
the right-hand side is well defined. We also appeal to the convention that
if $t+1-\upsilon $ is a pole of the Gamma function and $t+1$ is not a pole,
then $t^{\left( \upsilon \right) }=0$.
\end{definition}

\begin{definition}
The $\upsilon $-th fractional sum of a function $f$, for $\upsilon >0$ and $%
t\in \mathbb{N}_{a+\upsilon }$, is defined as \vskip -12pt 
\begin{equation*}
\Delta ^{-\upsilon }f(t)=\Delta ^{-\upsilon }f\left( t;a\right) :=\frac{1}{
\Gamma \left( \upsilon \right) }\sum_{s=a}^{t-\upsilon }\left( t-s-1\right)
^{\left( \upsilon -1\right) }f(s).
\end{equation*}

We also define the $\upsilon $-th fractional difference for $\upsilon >0$ by 
$\Delta ^{\upsilon }f(t):=\Delta ^{N}\Delta ^{\upsilon -N}f(t),$ where $t\in 
\mathbb{N}_{a+\upsilon }$ and $N\in \mathbb{N}$ is chosen so that $0\leq
N-1<\upsilon \leq N$.
\end{definition}

\begin{lemma}
\textrm{(\cite[Lemma 2.3]{goodrich})} Let $t$ and $\upsilon $ be any numbers
for which $t^{\left( \upsilon \right) } $ and $t^{\left( \upsilon -1\right) }
$ be defined. Then $\Delta t^{\left( \upsilon \right) }=\upsilon t^{\left(
\upsilon -1\right) }$.
\end{lemma}

\begin{lemma}
\textrm{(\cite[Lemma 2.1]{AE})} 
\begin{equation} \label{lemma2.1_4}
R_{\upsilon -1}\left( t^{\left( \upsilon -1\right) }\right) \left( s\right) =%
\frac{\Gamma \left( \upsilon \right) }{s^{\upsilon }}.
\end{equation}
\end{lemma}

\begin{lemma}
\textrm{(\cite[Lemma 2.2]{AE})} If $\mu >0$ and $m-1<\mu <m,$ where $m$
denotes a positive integer and $f$ is defined on $\mathbb{N}_{\mu -m},$ then
\begin{equation} \label{2.2_4}
R_{0}\left( \Delta _{\mu -m}^{\mu }f\right) \left( s\right) =s^{\mu }R_{\mu
-m}\left( f\right) \left( s\right) -\sum_{k=0}^{m-1}s^{m-k-1}\left. \left(
\Delta ^{k}\Delta _{\mu -m}^{-\left( \mu -m\right) }f\right) \right| _{t=0}.
\end{equation}
\end{lemma}

\begin{lemma}
\textrm{(\cite[Lemma 2.4]{AE})} 
\begin{equation} \label{lemma2.4_4}
R_{\upsilon -2}\left( f\ast _{\upsilon -2}g\right) \left( s\right) =\left(
s+1\right) ^{\upsilon -1}R_{\upsilon -2}\left( f\right) \left( s\right)
R_{\upsilon -2}\left( g\right) \left( s\right), 
\end{equation}
\end{lemma}
where 
\begin{equation} \label{eq2.4_4}
f\ast _{\upsilon -2}g\left( t\right) =\sum_{s=\upsilon -2}^{t}f\left(
t-s+\upsilon -2\right) g\left( s\right) 
\end{equation}
is the convolution product of two functions defined on $\mathbb{N}_{\upsilon -2}.$

\begin{definition}
The two parameter delta discrete Mittag-Leffler function is defined by 
\begin{equation*}
e_{\alpha, \beta}(\lambda, t - a) = \sum^{\infty}_{k = 0}\lambda^{k} \frac{%
(t - a + k \alpha + \beta - 1)^{(k\alpha + \beta - 1)}}{\Gamma(k \alpha +
\beta)},
\end{equation*}
for $\alpha > 0$, $\beta \in \mathbb{R}$ and $t \in \mathbb{N}_{a}$.
\end{definition}
Observe that, 
\begin{equation*}
e_{\alpha, \beta}(\lambda, -1) = \sum^{\infty}_{k = 0}\lambda^{k} \frac{(k
\alpha + \beta - 2)^{(k\alpha + \beta - 1)}}{\Gamma(k \alpha + \beta)} = 0.
\end{equation*}
Clearly, if $|\lambda| < 1$, then $e_{\alpha, \beta}(\lambda, 0) = \frac{1}{%
1 - \lambda}$. Also, 
\begin{equation*}
e_{\alpha, \beta}(\lambda, 1) = \sum^{\infty}_{k = 0}\lambda^{k} \frac{(k
\alpha + \beta)^{(k\alpha + \beta - 1)}}{\Gamma(k \alpha + \beta)} =
\sum^{\infty}_{k = 0}\lambda^{k} (k \alpha + \beta) = \frac{\alpha \lambda}{%
(1 - \lambda)^2} + \frac{\beta}{(1 - \lambda)}.
\end{equation*}
\begin{remark}
	\label{r-convergence}
Using D'Alembert's Ratio test, one can easily check that the above function converges for all
$|\lambda|<1$ and diverges for $|\lambda|>1$. \cite[Theorem 6]{shobanadevi}.
\end{remark} 

\section{Construction of the Green's Function}

In this section we will construct the Green's function related to Problem %
\eqref{LDE 1} - \eqref{e-Dirichlet}, following the approach given in \cite{CDJ}. To this end we use the main properties
of the Laplace transform. Since the delta discrete Mittag-Leffler function will be used, along all the section we assume that $|\alpha|<1$, in order to ensure its convergence by the characterization given in Remark \ref{r-convergence}.

First, from \eqref{2.2_4}, we have
\begin{equation} \label{1}
R_{0}\left[\Delta^{\upsilon}y(t)\right](s) = s^{\upsilon}R_{\upsilon - 2}\left[y(t)\right](s) - s A -  B,
\end{equation}
where $A = \left[\Delta^{\upsilon - 2}y(t)\right]_{t = 0}$ and $B = \left[\Delta^{\upsilon - 1}y(t)\right]_{t = 0}$. One can check that
\begin{equation} \label{A}
A = \frac{1}{\Gamma(2 - \upsilon)}(1 - \upsilon)^{(1 - \upsilon)}y(\upsilon - 2) = y(\upsilon - 2)
\end{equation}
and 
\begin{equation*} 
B  = (1 - \upsilon) y(\upsilon - 2) + y(\upsilon - 1).
\end{equation*}
Denote by $Y_{1}(t) = y(t + \upsilon - \mu)$. Then,
\begin{align} \label{2}
\nonumber R_{\mu - 1}\left[Y_{1}(t)\right](s) & = \sum^{\infty}_{t = \mu - 1}\left(\frac{1}{s + 1}\right)^{t + 1}Y_1(t) \\ \nonumber & = \sum^{\infty}_{t = \upsilon - 1}\left(\frac{1}{s + 1}\right)^{t - \upsilon + \mu + 1}y(t) \\ \nonumber & = (s + 1)^{\upsilon - \mu}\sum^{\infty}_{t = \upsilon - 1}\left(\frac{1}{s + 1}\right)^{t + 1}y(t) \\ \nonumber & = (s + 1)^{\upsilon - \mu} \left[\sum^{\infty}_{t = \upsilon - 2}\left(\frac{1}{s + 1}\right)^{t + 1}y(t) - \left(\frac{1}{s + 1}\right)^{\upsilon - 2 + 1}y(\upsilon - 2)\right] \\ & = (s + 1)^{\upsilon - \mu} R_{\upsilon - 2}\left[y(t)\right](s) - (s + 1)^{1 - \mu}y(\upsilon - 2).
\end{align}
Next, from \eqref{2.2_4}, we have
\begin{equation} \label{3}
R_{0}\left[\Delta^{\mu}y(t)\right](s) = s^{\mu}R_{\mu - 1}\left[y(t)\right](s) - \left[\Delta^{\mu - 1}y(t)\right]_{t = 0}.
\end{equation}

Using \eqref{2} and \eqref{3}, we obtain
\begin{align} \label{4}
\nonumber & R_{0}\left[\Delta^{\mu}Y_1(t)\right](s) \\ \nonumber & = s^{\mu}R_{\mu - 1}\left[Y_1(t)\right](s) - \left[\Delta^{\mu - 1}Y_1(t)\right]_{t = 0} \\ & = s^{\mu}\left[(s + 1)^{\upsilon - \mu} R_{\upsilon - 2}\left[y(t)\right](s) - (s + 1)^{1 - \mu}y(\upsilon - 2)\right] - \left[\Delta^{\mu - 1}Y_1(t)\right]_{t = 0}.
\end{align}
Now, consider
\begin{align}\label{5}
\nonumber  \left[\Delta^{\mu - 1}Y_1(t)\right]_{t = 0} & = \left[\Delta^{-(1 - \mu)}Y_1(t)\right]_{t = 0} \\ \nonumber & = \left[\frac{1}{\Gamma(1 - \mu)}\sum^{t - (1 - \mu)}_{s = \mu - 1}(t - s - 1)^{(1 - \mu - 1)}Y_{1}(s)\right]_{t = 0} \\ \nonumber & = \left[\frac{1}{\Gamma(1 - \mu)}\sum^{t - (1 - \mu)}_{s = \mu - 1}(t - s - 1)^{(1 - \mu - 1)}y(s + \upsilon - \mu)\right]_{t = 0} \\ \nonumber & = \frac{1}{\Gamma(1 - \mu)}(- \mu)^{(- \mu)}y(\upsilon - 1) \\ & = y(\upsilon - 1) = B - (1 - \upsilon)A.
\end{align}
Using \eqref{A} and \eqref{5} in \eqref{4}, we deduce
\begin{align} \label{6}
\nonumber & R_{0}\left[\Delta^{\mu}Y_1(t)\right](s) \\ \nonumber & = s^{\mu}(s + 1)^{\upsilon - \mu} R_{\upsilon - 2}\left[y(t)\right](s) - s^{\mu} (s + 1)^{1 - \mu}A - B + (1 - \upsilon)A \\ & = s^{\mu}(s + 1)^{\upsilon - \mu} R_{\upsilon - 2}\left[y(t)\right](s) + \left[(1 - \upsilon) - s^{\mu} (s + 1)^{1 - \mu}\right]A - B.
\end{align}
Denote $H_{1}(t) = h(t + \upsilon - 1)$. Then,
\begin{align} \label{7}
\nonumber R_{0}\left[H_{1}(t)\right](s) & = \sum^{\infty}_{t = 0}\left(\frac{1}{s + 1}\right)^{t + 1}H_1(t) \\ \nonumber & = \sum^{\infty}_{t = \upsilon - 1}\left(\frac{1}{s + 1}\right)^{t - \upsilon + 1 + 1}h(t) \\ \nonumber & = (s + 1)^{\upsilon - 1}\sum^{\infty}_{t = \upsilon - 1}\left(\frac{1}{s + 1}\right)^{t + 1}h(t) \\ \nonumber & = (s + 1)^{\upsilon - 1}\sum^{\infty}_{t = \upsilon - 2}\left(\frac{1}{s + 1}\right)^{t + 1}h(t) - h(\upsilon - 2) \\ & = (s + 1)^{\upsilon - 1} R_{\upsilon - 2}\left[h(t)\right](s) -  h(\upsilon - 2).
\end{align}
By applying $R_{0}$ to each side of \eqref{LDE 1} and employing \eqref{1}, \eqref{6} and \eqref{7}, we obtain
\begin{multline*}
- \left[s^{\upsilon}R_{\upsilon - 2}\left[y(t)\right](s) - s A -  B\right] \\ + \alpha \left[s^{\mu}(s + 1)^{\upsilon - \mu} R_{\upsilon - 2}\left[y(t)\right](s) + \left[(1 - \upsilon) - s^{\mu} (s + 1)^{1 - \mu}\right]A - B\right] \\ = (s + 1)^{\upsilon - 1} R_{\upsilon - 2}\left[h(t)\right](s) -  h(\upsilon - 2). 
\end{multline*}
Rearranging the terms gives us
\begin{multline*}
\left(s^{\upsilon} - \alpha s^{\mu}(s + 1)^{\upsilon - \mu} \right)R_{\upsilon - 2}\left[y(t)\right](s) \\ = \left(s + \alpha (1 - \upsilon) - \alpha s^{\mu} (s + 1)^{1 - \mu}\right)A + (1 - \alpha)B - (s + 1)^{\upsilon - 1} R_{\upsilon - 2}\left[h(t)\right](s) +  h(\upsilon - 2).
\end{multline*}
This implies that
\begin{multline} \label{Main}
R_{\upsilon - 2}\left[y(t)\right](s) = \frac{\left(s + \alpha (1 - \upsilon) - \alpha s^{\mu} (s + 1)^{1 - \mu}\right)}{\left(s^{\upsilon} - \alpha s^{\mu}(s + 1)^{\upsilon - \mu} \right)}A  + \frac{(1 - \alpha)}{\left(s^{\upsilon} - \alpha s^{\mu}(s + 1)^{\upsilon - \mu} \right)}B \\ - \frac{(s + 1)^{\upsilon - 1}}{\left(s^{\upsilon} - \alpha s^{\mu}(s + 1)^{\upsilon - \mu} \right)} R_{\upsilon - 2}\left[h(t)\right](s)  + \frac{1}{\left(s^{\upsilon} - \alpha s^{\mu}(s + 1)^{\upsilon - \mu} \right)}h(\upsilon - 2).
\end{multline}
Denote $Z(t) = y(t + n(\upsilon - \mu))$. Then,
\begin{align} \label{Z}
\nonumber R_{\upsilon - 2}\left[Z(t)\right](s) & = \sum^{\infty}_{t = \upsilon - 2}\left(\frac{1}{s + 1}\right)^{t + 1}Z(t) \\ \nonumber & = \sum^{\infty}_{t = (n + 1)\upsilon - n \mu - 2}\left(\frac{1}{s + 1}\right)^{t - n\upsilon + n\mu + 1}y(t) \\ \nonumber & = (s + 1)^{n\upsilon - n\mu}\sum^{\infty}_{t = (n + 1)\upsilon - n \mu - 2}\left(\frac{1}{s + 1}\right)^{t + 1}y(t) \\ & = (s + 1)^{n\upsilon - n\mu} R_{(n + 1)\upsilon - n \mu - 2}\left[y(t)\right](s).
\end{align}
Note that using \eqref{lemma2.1_4} and \eqref{Z}, we obtain
\begin{align} \label{One}
\nonumber \frac{s}{s^{\upsilon} - \alpha s^{\mu}(s + 1)^{\upsilon - \mu}} & = \frac{1}{s^{\upsilon - 1}} \frac{1}{\left[1 - \alpha \left(\frac{s + 1}{s}\right)^{\upsilon - \mu}\right]}  \\ \nonumber & = \frac{1}{s^{\upsilon - 1}} \sum^{\infty}_{k = 0}\alpha^k \left(\frac{s + 1}{s}\right)^{k\upsilon - k\mu} \\ \nonumber & = \sum^{\infty}_{k = 0}\alpha^k (s + 1)^{k\upsilon - k\mu} \frac{1}{s^{(k + 1) \upsilon - k \mu - 1}} \\ \nonumber & = \sum^{\infty}_{k = 0}\alpha^k (s + 1)^{k\upsilon - k\mu} \frac{R_{(k + 1) \upsilon - k \mu - 2}\left[t^{((k + 1) \upsilon - k \mu - 2)}\right](s)}{\Gamma((k + 1) \upsilon - k \mu - 1)}  \\ \nonumber & = \sum^{\infty}_{k = 0}\alpha^k \frac{R_{\upsilon - 2}\left[(t + k(\upsilon - \mu))^{((k + 1) \upsilon - k \mu - 2)}\right](s)}{\Gamma((k + 1) \upsilon - k \mu - 1)} \\ \nonumber & = R_{\upsilon - 2}\left[\sum^{\infty}_{k = 0}\alpha^k \frac{(t + k(\upsilon - \mu))^{((k + 1) \upsilon - k \mu - 2)}}{\Gamma((k + 1) \upsilon - k \mu - 1)}\right](s) \\ & =  R_{\upsilon - 2}\left[e_{\upsilon - \mu, \upsilon - 1}(\alpha, t - \upsilon + 2)\right](s).
\end{align}
Similar to \eqref{Z}, we have
\begin{equation} \label{Z A}
R_{\upsilon - 1}\left[Z(t)\right](s) = (s + 1)^{n\upsilon - n\mu} R_{(n + 1)\upsilon - n \mu - 1}\left[y(t)\right](s).
\end{equation}
Moreover, using \eqref{lemma2.1_4} and \eqref{Z A}  we obtain
\begin{align} \label{Two}
\nonumber \frac{1}{s^{\upsilon} - \alpha s^{\mu}(s + 1)^{\upsilon - \mu}} & = \frac{1}{s^{\upsilon}} \sum^{\infty}_{k = 0}\alpha^k \left(\frac{s + 1}{s}\right)^{k\upsilon - k\mu} \\ \nonumber & = \sum^{\infty}_{k = 0}\alpha^k (s + 1)^{k\upsilon - k\mu} \frac{1}{s^{(k + 1) \upsilon - k \mu}} \\ \nonumber & = \sum^{\infty}_{k = 0}\alpha^k (s + 1)^{k\upsilon - k\mu} \frac{R_{(k + 1) \upsilon - k \mu - 1}\left[t^{((k + 1) \upsilon - k \mu - 1)}\right](s)}{\Gamma((k + 1) \upsilon - k \mu)} \\ \nonumber & = \sum^{\infty}_{k = 0}\alpha^k \frac{R_{\upsilon - 1}\left[(t + k(\upsilon - \mu))^{((k + 1) \upsilon - k \mu - 1)}\right](s)}{\Gamma((k + 1) \upsilon - k \mu)}\\ \nonumber & = R_{\upsilon - 1}\left[\sum^{\infty}_{k = 0}\alpha^k \frac{(t + k(\upsilon - \mu))^{((k + 1) \upsilon - k \mu - 1)}}{\Gamma((k + 1) \upsilon - k \mu)}\right](s) \\ \nonumber & = R_{\upsilon - 2}\left[\sum^{\infty}_{k = 0}\alpha^k \frac{(t + k(\upsilon - \mu))^{((k + 1) \upsilon - k \mu - 1)}}{\Gamma((k + 1) \upsilon - k \mu)}\right](s) \\ \nonumber  & \quad - (s + 1)^{1 - \upsilon} \left[\sum^{\infty}_{k = 0}\alpha^k \frac{(t + k(\upsilon - \mu))^{((k + 1) \upsilon - k \mu - 1)}}{\Gamma((k + 1) \upsilon - k \mu)}\right]_{t = \upsilon - 2}  \\ \nonumber & = R_{\upsilon - 2}\left[\sum^{\infty}_{k = 0}\alpha^k \frac{(t + k(\upsilon - \mu))^{((k + 1) \upsilon - k \mu - 1)}}{\Gamma((k + 1) \upsilon - k \mu)}\right](s) \\ & =  R_{\upsilon - 2}\left[e_{\upsilon - \mu, \upsilon}(\alpha, t - \upsilon + 1)\right](s).
\end{align}
Denote $Z_{1}(t) = y(t + n(\upsilon - \mu) - \mu + 1)$. Then,
\begin{align} \label{Z B}
\nonumber R_{\upsilon - 2}\left[Z_1(t)\right](s) & = \sum^{\infty}_{t = \upsilon - 2}\left(\frac{1}{s + 1}\right)^{t + 1}Z_1(t) \\ \nonumber & = \sum^{\infty}_{t = (n + 1)(\upsilon - \mu) - 1}\left(\frac{1}{s + 1}\right)^{t - n(\upsilon - \mu) + \mu - 1 + 1}y(t) \\ \nonumber & = (s + 1)^{n\upsilon - (n + 1)\mu + 1}\sum^{\infty}_{t = (n + 1)(\upsilon - \mu) - 1}\left(\frac{1}{s + 1}\right)^{t + 1}y(t) \\ & = (s + 1)^{n\upsilon - (n + 1)\mu + 1} R_{(n + 1)(\upsilon - \mu) - 1}\left[y(t)\right](s).
\end{align}
Note that using \eqref{lemma2.1_4} and \eqref{Z B}, we obtain
\begin{align} \label{Three}
\nonumber & \frac{s^{\mu} (s + 1)^{1 - \mu}}{s^{\upsilon} - \alpha s^{\mu}(s + 1)^{\upsilon - \mu}} \\ \nonumber & = \frac{s^{\mu} (s + 1)^{1 - \mu}}{s^{\upsilon}} \frac{1}{1 - \alpha s^{\mu - \upsilon}(s + 1)^{\upsilon - \mu}}  \\ \nonumber & = \frac{s^{\mu} (s + 1)^{1 - \mu}}{s^{\upsilon}} \frac{1}{\left[1 - \alpha \left(\frac{s + 1}{s}\right)^{\upsilon - \mu}\right]}  \\ \nonumber & = \frac{s^{\mu} (s + 1)^{1 - \mu}}{s^{\upsilon}} \sum^{\infty}_{k = 0}\alpha^k \left(\frac{s + 1}{s}\right)^{k\upsilon - k\mu} \\ \nonumber & = \sum^{\infty}_{k = 0}\alpha^k (s + 1)^{k\upsilon - (k + 1)\mu + 1} \frac{1}{s^{(k + 1) \upsilon - (k + 1) \mu}} \\ \nonumber & = \sum^{\infty}_{k = 0}\alpha^k (s + 1)^{k\upsilon - (k + 1)\mu + 1} \frac{R_{(k + 1) \upsilon - (k + 1) \mu - 1}\left[t^{((k + 1) \upsilon - (k + 1) \mu - 1)}\right](s)}{\Gamma((k + 1) \upsilon - (k + 1) \mu)}  \\ \nonumber & = \sum^{\infty}_{k = 0}\alpha^k \frac{R_{\upsilon - 2}\left[(t + k(\upsilon - \mu) - \mu + 1)^{((k + 1) \upsilon - k \mu - 1)}\right](s)}{\Gamma((k + 1) \upsilon - k \mu)} \\ \nonumber & = R_{\upsilon - 2}\left[\sum^{\infty}_{k = 0}\alpha^k \frac{(t + k(\upsilon - \mu) - \mu + 1)^{((k + 1) \upsilon - k \mu - 1)}}{\Gamma((k + 1) \upsilon - k \mu)}\right](s) \\ & =  R_{\upsilon - 2}\left[e_{\upsilon - \mu, \upsilon - \mu}(\alpha, t - \upsilon + 2)\right](s).
\end{align}
Using \eqref{One}, \eqref{Two} and \eqref{Three} in \eqref{Main}, we deduce
\begin{multline*} 
R_{\upsilon - 2}\left[y(t)\right](s) = \Big{[}R_{\upsilon - 2}\left[e_{\upsilon - \mu, \upsilon - 1}(\alpha, t - \upsilon + 2)\right](s) \\ + \alpha (1 - \upsilon)R_{\upsilon - 2}\left[e_{\upsilon - \mu, \upsilon}(\alpha, t - \upsilon + 1)\right](s) \\ - \alpha R_{\upsilon - 2}\left[e_{\upsilon - \mu, \upsilon - \mu}(\alpha, t - \upsilon + 2)\right](s)\Big{]} A \\ + (1 - \alpha) R_{\upsilon - 2}\left[e_{\upsilon - \mu, \upsilon}(\alpha, t - \upsilon + 1)\right](s)B \\ - (s + 1)^{\upsilon - 1} R_{\upsilon - 2}\left[e_{\upsilon - \mu, \upsilon}(\alpha, t - \upsilon + 1)\right](s)R_{\upsilon - 2}\left[h(t)\right](s) \\ + R_{\upsilon - 2}\left[e_{\upsilon - \mu, \upsilon}(\alpha, t - \upsilon + 1)\right](s)h(\upsilon - 2),
\end{multline*}
which, by \eqref{lemma2.4_4} can be written as
\begin{multline*} 
R_{\upsilon - 2}\left[y(t)\right](s) = \Big{[}R_{\upsilon - 2}\left[e_{\upsilon - \mu, \upsilon - 1}(\alpha, t - \upsilon + 2)\right](s) \\ + \alpha (1 - \upsilon)R_{\upsilon - 2}\left[e_{\upsilon - \mu, \upsilon}(\alpha, t - \upsilon + 1)\right](s) \\ - \alpha R_{\upsilon - 2}\left[e_{\upsilon - \mu, \upsilon - \mu}(\alpha, t - \upsilon + 2)\right](s)\Big{]} A \\ + (1 - \alpha) R_{\upsilon - 2}\left[e_{\upsilon - \mu, \upsilon}(\alpha, t - \upsilon + 1)\right](s)B \\ - R_{\upsilon - 2}\left[e_{\upsilon - \mu, \upsilon}(\alpha, t - \upsilon + 1) *_{\upsilon  - 2} h\right](s) \\ + R_{\upsilon - 2}\left[e_{\upsilon - \mu, \upsilon}(\alpha, t - \upsilon + 1)\right](s)h(\upsilon - 2).
\end{multline*}
Apply to each side the inverse of $R_{\upsilon - 2}$, we obtain
\begin{multline*} 
y(t) = \Big{[}e_{\upsilon - \mu, \upsilon - 1}(\alpha, t - \upsilon + 2) + \alpha (1 - \upsilon)e_{\upsilon - \mu, \upsilon}(\alpha, t - \upsilon + 1) \\ - \alpha e_{\upsilon - \mu, \upsilon - \mu}(\alpha, t - \upsilon + 2)\Big{]} A + (1 - \alpha) e_{\upsilon - \mu, \upsilon}(\alpha, t - \upsilon + 1)B \\ - e_{\upsilon - \mu, \upsilon}(\alpha, t - \upsilon + 1) *_{\upsilon  - 2} h + e_{\upsilon - \mu, \upsilon}(\alpha, t - \upsilon + 1)h(\upsilon - 2).
\end{multline*}
Thus, using \eqref{eq2.4_4}, we have
\begin{multline*} 
y(t) = \Big{[}e_{\upsilon - \mu, \upsilon - 1}(\alpha, t - \upsilon + 2) + \alpha (1 - \upsilon)e_{\upsilon - \mu, \upsilon}(\alpha, t - \upsilon + 1) \\ - \alpha e_{\upsilon - \mu, \upsilon - \mu}(\alpha, t - \upsilon + 2)\Big{]} A + (1 - \alpha) e_{\upsilon - \mu, \upsilon}(\alpha, t - \upsilon + 1)B \\ - \sum^{t}_{s = \upsilon - 2}e_{\upsilon - \mu, \upsilon}(\alpha, t - s + \upsilon - 2 - \upsilon + 1) h(s)  +  e_{\upsilon - \mu, \upsilon}(\alpha, t - \upsilon + 1)h(\upsilon - 2).
\end{multline*}
That is, 
\begin{multline} \label{Main 5}
y(t) = \Big{[}e_{\upsilon - \mu, \upsilon - 1}(\alpha, t - \upsilon + 2) + \alpha (1 - \upsilon)e_{\upsilon - \mu, \upsilon}(\alpha, t - \upsilon + 1) \\ - \alpha e_{\upsilon - \mu, \upsilon - \mu}(\alpha, t - \upsilon + 2)\Big{]} A + (1 - \alpha) e_{\upsilon - \mu, \upsilon}(\alpha, t - \upsilon + 1)B \\ - \sum^{t}_{s = \upsilon - 1}e_{\upsilon - \mu, \upsilon}(\alpha, t - s - 1) h(s).
\end{multline}
Using $y(\upsilon - 2) = 0$ in \eqref{Main 5}, we have
\begin{multline*} 
0 = \Big{[}e_{\upsilon - \mu, \upsilon - 1}(\alpha, 0) + \alpha (1 - \upsilon)e_{\upsilon - \mu, \upsilon}(\alpha, -1) \\ - \alpha e_{\upsilon - \mu, \upsilon - \mu}(\alpha, 0)\Big{]} A + (1 - \alpha) e_{\upsilon - \mu, \upsilon}(\alpha, -1)B \\ - \sum^{\upsilon - 2}_{s = \upsilon - 1}e_{\upsilon - \mu, \upsilon}(\alpha, t - s - 1) h(s).
\end{multline*}
That is, 
\begin{equation*} 
0 = \left[\frac{1}{(1 - \alpha)} - \frac{\alpha}{(1 - \alpha)}\right] A.
\end{equation*}
Using $y(\upsilon + b + 1) = 0$ in \eqref{Main 5} and taking $A = 0$, we have
\begin{equation*}
0 = (1 - \alpha) e_{\upsilon - \mu, \upsilon}(\alpha, b + 2)B - \sum^{\upsilon + b + 1}_{s = \upsilon - 1}e_{\upsilon - \mu, \upsilon}(\alpha, \upsilon + b - s) h(s),
\end{equation*}
or
\begin{align} \label{B 1}
\nonumber B & = \frac{1}{(1 - \alpha) e_{\upsilon - \mu, \upsilon}(\alpha, b + 2)}\sum^{\upsilon + b + 1}_{s = \upsilon - 1}e_{\upsilon - \mu, \upsilon}(\alpha, \upsilon + b - s) h(s) \\ & = \frac{1}{(1 - \alpha) e_{\upsilon - \mu, \upsilon}(\alpha, b + 2)}\sum^{\upsilon + b}_{s = \upsilon - 1}e_{\upsilon - \mu, \upsilon}(\alpha, \upsilon + b - s) h(s).
\end{align}
Using \eqref{B 1} and $A = 0$ in \eqref{Main 5}, we obtain
\begin{multline*} 
y(t) = (1 - \alpha) e_{\upsilon - \mu, \upsilon}(\alpha, t - \upsilon + 1)\left[\frac{1}{(1 - \alpha) e_{\upsilon - \mu, \upsilon}(\alpha, b + 2)}\sum^{\upsilon + b}_{s = \upsilon - 1}e_{\upsilon - \mu, \upsilon}(\alpha, \upsilon + b - s) h(s)\right] \\ - \sum^{t}_{s = \upsilon - 1}e_{\upsilon - \mu, \upsilon}(\alpha, t - s - 1) h(s).
\end{multline*}
Rearranging the terms, we obtain
\begin{multline*} 
y(t) = \frac{e_{\upsilon - \mu, \upsilon}(\alpha, t - \upsilon + 1)}{e_{\upsilon - \mu, \upsilon}(\alpha, b + 2)}\sum^{\upsilon + b}_{s = \upsilon - 1}e_{\upsilon - \mu, \upsilon}(\alpha, \upsilon + b - s) h(s)\\ - \sum^{t}_{s = \upsilon - 1}e_{\upsilon - \mu, \upsilon}(\alpha, t - s - 1) h(s).
\end{multline*}
That is, 
\begin{multline*} \label{Main 11}
y(t) = \sum^{t}_{s = \upsilon - 1}\left[\frac{e_{\upsilon - \mu, \upsilon}(\alpha, t - \upsilon + 1)}{e_{\upsilon - \mu, \upsilon}(\alpha, b + 2)} e_{\upsilon - \mu, \upsilon}(\alpha, \upsilon + b - s) - e_{\upsilon - \mu, \upsilon}(\alpha, t - s - 1) \right] h(s)\\ - \sum^{\upsilon + b}_{s = t + 1}\left[\frac{e_{\upsilon - \mu, \upsilon}(\alpha, t - \upsilon + 1)}{e_{\upsilon - \mu, \upsilon}(\alpha, b + 2)} e_{\upsilon - \mu, \upsilon}(\alpha, \upsilon + b - s) \right] h(s).
\end{multline*}

\begin{theorem}
	Assuming that $|\alpha|<1$, we have that Problem 
\eqref{LDE 1} - \eqref{e-Dirichlet} has a unique solution if and only if $$e_{\upsilon - \mu, \upsilon}(\alpha, b + 2) \neq 0.$$
\end{theorem}

Denote by $$I_1 = \{(t, s) : \upsilon - 1 \leq s \leq t \leq \upsilon + b + 1\},$$ and $$I_2 = \{(t, s) : \upsilon - 1 \leq t + 1 \leq s \leq \upsilon + b\}.$$ Then, the expression for the related Green's function, when $|\alpha|<1$, is given by
\begin{equation*} \label{Greens}
G(t, s) = 
\begin{cases}
\frac{e_{\upsilon - \mu, \upsilon}(\alpha, t - \upsilon + 1)}{e_{\upsilon - \mu, \upsilon}(\alpha, b + 2)} e_{\upsilon - \mu, \upsilon}(\alpha, \upsilon + b - s) - e_{\upsilon - \mu, \upsilon}(\alpha, t - s - 1), ~ (t, s) \in I_1, \\
\frac{e_{\upsilon - \mu, \upsilon}(\alpha, t - \upsilon + 1)}{e_{\upsilon - \mu, \upsilon}(\alpha, b + 2)} e_{\upsilon - \mu, \upsilon}(\alpha, \upsilon + b - s), \hspace{1.4 in} (t, s) \in I_2.
\end{cases}
\end{equation*}

\section{Existence of Solutions of Nonlinear Problems}
In this section we will apply the following Krasnosel'skii--Zabreiko fixed point theorem to obtain nontrivial solutions of 
\begin{equation}  \label{LDE NL}
- \Delta^{\upsilon}y(t) + \alpha \Delta^{\mu}y(t + \upsilon - \mu) = f(t + \upsilon - 1, y(t + \upsilon - 1)), \quad t \in I,
\end{equation}
coupled to the boundary conditions \eqref{e-Dirichlet}. 

Here we assume that $f : [\upsilon - 1, \upsilon + b]_{\mathbb{N}_{\upsilon - 1}} \times \mathbb{R} \rightarrow \mathbb{R}$ is a continuous function.


\begin{theorem} \label{KZ}
Let $X$ be a Banach space and $F : X \rightarrow X$ be a completely continuous operator. If there exists a bounded linear operator $A : X \rightarrow X$ such that $1$ is not an eigenvalue and $$\lim_{\left\|y\right\| \rightarrow \infty} \frac{\|F(y) - A(y)\|}{\|y\|} = 0,$$ then $F$ has a fixed point in $X$.
\end{theorem}

We will apply Theorem \ref{KZ} to a nonlinear summation operator whose kernel is $G(t, s)$. The arguments are in the line to the ones used in \cite{henderson}. 

In this context, let the Banach space $(X, \| \cdot \|)$ be defined by
\begin{equation*} \label{X}
X : = \{h : [\upsilon - 2, \upsilon + b + 1]_{\mathbb{N}_{\upsilon - 2}} \rightarrow \mathbb{R}\},
\end{equation*}
with norm
\begin{equation*} \label{N}
\|h\| : = \max_{t \in [\upsilon - 2, \upsilon + b + 1]_{\mathbb{N}_{\upsilon - 2}}}\left|h(t)\right|.
\end{equation*}

Clearly, $y \in X$ is a fixed point of the completely continuous operator $F : X \rightarrow X$ defined by
\begin{equation*} \label{F}
\left(Fy\right)(t) : = \sum^{\upsilon + b}_{s = \upsilon - 1}G(t, s) f(s, y(s)), \quad t \in [\upsilon - 2, \upsilon + b + 1]_{\mathbb{N}_{\upsilon - 2}}.
\end{equation*}

We now apply Theorem \ref{KZ} to the operator $F$ defined in \eqref{F} and to an associated linear operator in establishing solutions of \eqref{LDE NL}--\eqref{e-Dirichlet}.

\begin{theorem} \label{KZ 1}
Assume that $|\alpha|<1$ and $f : [\upsilon - 1, \upsilon + b]_{\mathbb{N}_{\upsilon - 1}} \times \mathbb{R} \rightarrow \mathbb{R}$ is continuous and $$\lim_{\left|r\right| \rightarrow \infty} \frac{f(t + \upsilon - 1, r)}{r} = m$$ for all $t \in I$.  

If $$\left|m\right| < d : = \frac{1}{\displaystyle{\max_{t \in [\upsilon - 2, \upsilon + b + 1]_{\mathbb{N}_{\upsilon - 2}}}\sum^{\upsilon + b}_{s = \upsilon - 1}\left|G(t, s)\right|}},$$ then the boundary value problem \eqref{LDE NL}--\eqref{e-Dirichlet} has a solution $y$, and moreover, $y \not \equiv 0$ on $[\upsilon - 2, \upsilon + b + 1]_{\mathbb{N}_{\upsilon - 2}}$, when $f(t, 0) \neq 0$ for at least one  $t \in I$. 
\end{theorem}

\begin{proof}

Corresponding to \eqref{LDE NL}--\eqref{e-Dirichlet}, we consider the following linear equation 
\begin{equation}  \label{LDE L 1}
- \Delta^{\upsilon}y(t) + \alpha \Delta^{\mu}y(t + \upsilon - \mu) = m \,y(t + \upsilon - 1), \quad t \in I, 
\end{equation}
coupled to the boundary conditions \eqref{e-Dirichlet}.  We define a completely continuous linear operator $A : X \rightarrow X$ by 
\begin{equation*} \label{AA}
\left(Ay\right)(t) : = m \sum^{\upsilon + b}_{s = \upsilon - 1}G(t, s) y(s), \quad t \in [\upsilon - 2, \upsilon + b + 1]_{\mathbb{N}_{\upsilon - 2}}.
\end{equation*}
Clearly, solutions of \eqref{LDE L 1}--\eqref{e-Dirichlet} are fixed points of $A$, and conversely. 

First, we show that $1$ is not an eigenvalue of $A$. To see this, we consider two cases: (a) $m = 0$ and (b) $m \neq 0$.

For (a), if $m = 0$, since the boundary value problem \eqref{LDE L 1}--\eqref{e-Dirichlet} has only the trivial solution, it is immediately that $1$ is not an eigenvalue of $A$.

For (b), if $m \neq 0$ and \eqref{LDE L 1}--\eqref{e-Dirichlet} has a nontrivial solution, then $\left\|y\right\| > 0$. And so, we have
\begin{align*}
\left\|y\right\| & = \left\|\left(Ay\right)\right\| \\ & = \max_{t \in [\upsilon - 2, \upsilon + b + 1]_{\mathbb{N}_{\upsilon - 2}}} \left|m \sum^{\upsilon + b}_{s = \upsilon - 1}G(t, s) y(s)\right| \\ & = |m| \max_{t \in [\upsilon - 2, \upsilon + b + 1]_{\mathbb{N}_{\upsilon - 2}}} \left| \sum^{\upsilon + b}_{s = \upsilon - 1}G(t, s) y(s)\right| \\ & \leq |m| \left\|y\right\| \max_{t \in [\upsilon - 2, \upsilon + b + 1]_{\mathbb{N}_{\upsilon - 2}}}  \sum^{\upsilon + b}_{s = \upsilon - 1}\left|G(t, s)\right| \\ & < d \left\|y\right\| \frac{1}{d} \\ & = \left\|y\right\|,
\end{align*}
a contradiction. Again, $1$ is not an eigenvalue of $A$. 

Our next claim is that $$\lim_{\left\|y\right\| \rightarrow \infty} \frac{\|F(y) - A(y)\|}{\|y\|} = 0.$$ In this direction, let $\varepsilon > 0$ be given. Now, let $$\lim_{\left|r\right| \rightarrow \infty} \frac{f(t + \upsilon - 1, r)}{r} = m$$ for all $t \in I$, implies that there exists an $N_1 > 0$ such that, for $|r| > N_1$, 
\begin{equation} \label{F 1}
\left|f(t + \upsilon - 1, r) - m r\right| < \varepsilon \left|r\right|, \quad t \in I.
\end{equation}
Let $$N = \sup_{\left|r\right| \leq N_1, \; t \in I} \left|f(t + \upsilon - 1, r)\right|,$$ and let $L \geq N_1$ be such that $$\frac{N + |m| N_1}{L} < \varepsilon.$$ Next, choose $y \in X$ with $\left\|y\right\| > L$. Now, for $s \in [\upsilon - 2, \upsilon + b + 1]_{\mathbb{N}_{\upsilon - 2}}$, 
if $|y(s)| \leq N_1$, we have
\begin{align*}
\left|f(s, y(s)) - m y(s)\right|  \leq \left|f(s, y(s))\right| + \left|m\right| \left|y(s)\right|  \leq N + \left|m\right| N_1  < \varepsilon L  < \varepsilon \left\|y\right\|.
\end{align*}
On the other hand, if $|y(s)| > N_1$, we have from \eqref{F 1} that $$\left|f(s, y(s)) - m y(s)\right| < \varepsilon |y(s)| \leq \varepsilon \left\|y\right\|.$$
Thus, for $s \in [\upsilon - 2, \upsilon + b + 1]_{\mathbb{N}_{\upsilon - 2}}$, 
\begin{equation} \label{F 2}
\left|f(s, y(s)) - m y(s)\right| \leq \varepsilon \left\|y\right\|.
\end{equation}
It follows from \eqref{F 2} that, for $y \in X$ with $\left\|y\right\| > L$,
\begin{align*}
\|F(y) - A(y)\| & = \max_{t \in [\upsilon - 2, \upsilon + b + 1]_{\mathbb{N}_{\upsilon - 2}}} \left|\sum^{\upsilon + b}_{s = \upsilon - 1}G(t, s) [f(s, y(s)) - m y(s)]\right| \\ & \leq \max_{t \in [\upsilon - 2, \upsilon + b + 1]_{\mathbb{N}_{\upsilon - 2}}}  \sum^{\upsilon + b}_{s = \upsilon - 1}\left|G(t, s)\right| \left|f(s, y(s)) - m y(s)\right| \\ & \leq \varepsilon \left\|y\right\| \max_{t \in [\upsilon - 2, \upsilon + b + 1]_{\mathbb{N}_{\upsilon - 2}}}  \sum^{\upsilon + b}_{s = \upsilon - 1}\left|G(t, s)\right| \\ & = \varepsilon \left\|y\right\| \frac{1}{d}.
\end{align*}
Therefore, $$\lim_{\left\|y\right\| \rightarrow \infty} \frac{\|F(y) - A(y)\|}{\|y\|} = 0.$$ By Theorem \ref{KZ}, $F$ has a fixed point $y \in X$, and $y$ is a desired solution of \eqref{LDE NL}, \eqref{e-Dirichlet}. Moreover $y \not \equiv 0$ on $[\upsilon - 2, \upsilon + b + 1]_{\mathbb{N}_{\upsilon - 2}}$, when $f(t, 0) \neq 0$ for at least one  $t \in I$ and the proof is complete.
\end{proof}

Let us recall the following theorem
\begin{theorem} \cite{Ag} (Leray--Schauder Nonlinear Alternative) \label{LS}
Let $(E, \|\cdot\|)$ be a Banach space, $K$ be a closed and convex subset of $E$, $U$ be a relatively open subset of $K$ such that $0 \in U$, and $T : \bar{U} \rightarrow K$ be completely continuous. Then, either
\begin{enumerate}
\item[(i)] $u = Tu$ has a solution in $\bar{U}$; or
\item[(ii)] There exist $u \in \partial U$ and $\lambda \in (0, 1)$ such that $u = \lambda Tu$.
\end{enumerate}
\end{theorem}
Our second main result in this paper is as follows.
\begin{theorem} \label{LS 1}
Assume that $|\alpha|<1$ and that the following properties are fulfilled:
\begin{enumerate}
\item[(A1)] There exist a nondecreasing function $\psi : [0, \infty) \rightarrow [0, \infty)$ and $g : [\upsilon - 2, \upsilon + b + 1]_{\mathbb{N}_{\upsilon - 2}} \rightarrow [0, \infty)$ such that $$\left|f(t + \upsilon - 1, r)\right| \leq g(t + \upsilon - 1) \psi\left(|r|\right), \quad t \in I, \quad r \in \mathbb{R}.$$
\item[(A2)] There exists $L > 0$ such that $$\displaystyle\frac{L}{\psi\left(L\right) {\displaystyle\max_{t \in [\upsilon - 2, \upsilon + b + 1]_{\mathbb{N}_{\upsilon - 2}}}\sum^{\upsilon + b}_{s = \upsilon - 1}g(s)\,\left|G(t, s)\right|}} > 1.$$
\end{enumerate}
Then, the boundary value problem \eqref{LDE NL}--\eqref{e-Dirichlet} has a solution defined on $[\upsilon - 2, \upsilon + b + 1]_{\mathbb{N}_{\upsilon - 2}}$. Moreover, $y \not \equiv 0$ on $[\upsilon - 2, \upsilon + b + 1]_{\mathbb{N}_{\upsilon - 2}}$, when $f(t, 0) \neq 0$ for at least one  $t \in I$. 
\end{theorem}

\begin{proof}
We first show that $F$ maps bounded sets into bounded sets. For this purpose, for $r > 0$, let $$B_r = \{y \in X : \|y\| \leq r\},$$ be a bounded subset of $X$. Then, by (A1), for $y \in B_r$, 
\begin{align*}
\left|\big{(}Fy\big{)}(t)\right| & \leq \sum^{\upsilon + b}_{s = \upsilon - 1}\left|G(t, s)\right| \left|f(s, y(s))\right| \leq \psi\left(\|y\|\right) \sum^{\upsilon + b}_{s = \upsilon - 1}g(s)\left|G(t, s)\right|,
\end{align*}
implying that $$\left|\big{(}Fy\big{)}(t)\right| \leq \psi\left(r\right) \sum^{\upsilon + b}_{s = \upsilon - 1}g(s)\left|G(t, s)\right|.$$ Thus, $F$ maps $B_r$ into a bounded set. Since $[\upsilon - 2, \upsilon + b + 1]_{\mathbb{N}_{\upsilon - 2}}$ is a discrete set, it follows immediately that $F$ maps $B_r$ into an equicontinuous set. Therefore, by the Arzela--Ascoli theorem, $F$ is completely continuous. Next, suppose that $y \in X$ and $y = \lambda F y$ for some $0 < \lambda < 1$. Then, from (A1), for $t \in [\upsilon - 2, \upsilon + b + 1]_{\mathbb{N}_{\upsilon - 2}}$, we have
\begin{align*}
\left|y(t)\right| = \left|\lambda \big{(}Fy\big{)}(t)\right| \leq \sum^{\upsilon + b}_{s = \upsilon - 1}\left|G(t, s)\right| \left|f(s, y(s))\right|  \leq \psi\left(\|y\|\right) \sum^{\upsilon + b}_{s = \upsilon - 1}g(s)\left|G(t, s)\right|,
\end{align*}
implying that $$\displaystyle\frac{\|y\|}{\psi\left(\|y\|\right) \displaystyle \sum^{\upsilon + b}_{s = \upsilon - 1}g(s)\left|G(t, s)\right|} \leq 1.$$ It follows from (A2) that $\|y\| \neq L$. If we set $$U = \Big\{y \in X : \|y\| < L \Big\},$$ then the operator $F : \bar{U} \rightarrow X$ is completely continuous. From the choice of $U$, then there is no $y \in \partial U$ such that $y = \lambda F y$ for some $0 < \lambda < 1$. It follows from Theorem \ref{LS} that $F$ has a fixed point $y_0 \in \bar{U}$, which is a desired solution of \eqref{LDE NL}--\eqref{e-Dirichlet}. 

Obviously, if $f(t, 0) \not \equiv 0$ on $[\upsilon - 2, \upsilon + b + 1]_{\mathbb{N}_{\upsilon - 2}}$, we have that $y \not \equiv 0$ on $[\upsilon - 2, \upsilon + b + 1]_{\mathbb{N}_{\upsilon - 2}}$.
\end{proof}

As a direct consequence of the previous result, we deduce the following corollary:

\begin{corollary} \label{Cor-1}
	Assume that $|\alpha|<1$ and condition $(A1)$ in Theorem \ref{LS 1} holds. Then, if  
	$$\displaystyle\lim_{L \to +\infty}\frac{L}{\psi\left(L\right)}=+\infty$$
the boundary value problem \eqref{LDE NL}--\eqref{e-Dirichlet} has a solution defined on $[\upsilon - 2, \upsilon + b + 1]_{\mathbb{N}_{\upsilon - 2}}$. Moreover, $y \not \equiv 0$ on $[\upsilon - 2, \upsilon + b + 1]_{\mathbb{N}_{\upsilon - 2}}$, when $f(t, 0) \neq 0$ for at least one  $t \in I$. 
\end{corollary}

\section{Examples}
In this section, we provide two example to demonstrate the applicability of Theorems \ref{KZ 1} and \ref{LS 1}.

\begin{example}
Consider the boundary value problem \eqref{LDE NL}, \eqref{e-Dirichlet} with $a = 0$, $b = 5$, $\upsilon = 1.5$, $\mu = 0.5$, $\alpha = 0.5$ and 
$$f(t + \upsilon - 1, r) =  \frac{r}{3\pi} \left|\tan^{-1}\left((t + \upsilon - 1)^2 (r+1)^3\right)\right|+e^{(t + \upsilon - 1)^2}\,\sqrt{|r+1|}.$$

 Clearly, $$m = \lim_{\left|r\right| \rightarrow \infty} \frac{f(t + \upsilon - 1, r)}{r} = \frac{1}{6}\qquad \mbox{ for all $t \in I$.}$$
 
  The Green's function associated with the boundary value problem is given by 
\begin{equation} \label{Greens}
G(t, s) = 
\begin{cases}
\frac{e_{1, 1.5}(0.5, t - 0.5)}{e_{1, 1.5}(0.5, 7)} e_{1, 1.5}(0.5, 6.5 - s) - e_{1, 1.5}(0.5, t - s - 1), ~ (t, s) \in I_1, \\
\frac{e_{1, 1.5}(0.5, t - 0.5)}{e_{1, 1.5}(0.5, 7)} e_{1, 1.5}(0.5, 6.5 - s), \hspace{1.4 in} (t, s) \in I_2,
\end{cases}
\end{equation}
where $$I_1 = \{(t, s) : 0.5 \leq s \leq t \leq 7.5\},$$ and $$I_2 = \{(t, s) : 0.5 \leq t + 1 \leq s \leq 6.5\}.$$ Since $$d = \displaystyle\frac{1}{\displaystyle\max_{t \in [-0.5, 7.5]_{\mathbb{N}_{-0.5}}}\sum^{6.5}_{s = 0.5}\left|G(t, s)\right|} = 0.241342 > |m|,$$ by Theorem \ref{KZ 1}, the boundary value problem has a non trivial solution defined on $[-0.5, 7.5]_{\mathbb{N}_{-0.5}}$.
\end{example}

\begin{example}
Consider the boundary value problem \eqref{LDE NL}, \eqref{e-Dirichlet} with $a = 0$, $b = 5$, $\upsilon = 1.5$, $\mu = 0.5$, $\alpha = 0.5$ and $f(t + \upsilon - 1, r) = (t + \upsilon - 1)\,\sqrt[4]{|r|^3+t + \upsilon - 1}$. Clearly, $$\left|f(t + \upsilon - 1, r)\right| \leq g(t + \upsilon - 1)\psi\left(|r|\right), \quad t \in I, \quad r \in \mathbb{R},$$ where $$g(t + \upsilon - 1) = t + \upsilon - 1, \quad t \in I,$$ and $$\psi\left(|r|\right) = \sqrt[4]{|r|^3+b + \upsilon}, \quad r \in \mathbb{R}.$$ Also, $g : [\upsilon - 2, \upsilon + b + 1]_{\mathbb{N}_{\upsilon - 2}} \rightarrow [0, \infty)$ and $\psi : [0, \infty) \rightarrow [0, \infty)$ is a nondecreasing function. Thus, the assumption (A1) of Theorem \ref{LS 1} holds. 

Now, since 
$$\displaystyle\lim_{L \to +\infty}\frac{L}{\psi\left(L\right)}=+\infty,$$
from Corollary \ref{Cor-1}, we have that the considered problem has at least a nontrivial solution $y$.

Notice that, in this case, we can estimate the minimum value of the parameter $L$ as 74,395.4. We point out that, from the proof of Theorem \ref{LS 1}, this value of $L$ is the better a priori bound that we have for $\|y\|$.
\end{example}
 
\noindent{\bf Acknowledgments} {First author is partially supported by Xunta de Galicia (Spain), project EM2014/032 and AIE, Spain and FEDER, grant PID2020-113275GB-I00. The second author is supported by Project FNSE-03.}

\end{document}